\documentclass[11pt]{article}

\usepackage[a4paper,margin=1in]{geometry}

\usepackage{amsmath}
\usepackage{amssymb}
\usepackage{amsthm}

\usepackage{graphicx}
\usepackage{subcaption}

\usepackage{authblk}

\usepackage[
    backend=biber,
    style=numeric-comp,
    doi=true,
    sorting=none
]{biblatex}
\newtheorem{theorem}{Theorem}
\newtheorem{lemma}[theorem]{Lemma}

\theoremstyle{definition}
\newtheorem{assumption}{Assumption}

\title{\textbf{The inverse elastic scattering problem in the cavity
with impedance boundary}}

\author[1]{Shiyu Xu}
\author[1]{Huiling Zheng}
\author[1]{Fang Zeng\thanks{Corresponding author: fzeng@cqnu.edu.cn}}

\affil[1]{School of Mathematical Sciences,
Chongqing Normal University,
No. 37 University Town Middle Road,
Shapingba District, Chongqing 401331, China}

\date{}

\begin{document}

\maketitle

\begin{abstract}
In this paper, we investigate the elastic inverse scattering problem in a two-dimensional cavity with an impedance boundary condition. We prove that the scattered fields inside the cavity satisfy a reciprocity relation, and prove that both the shape of cavity and the impedance parameter on the boundary can be uniquely determined from the scattered fields measured on an interior curve. Numerically, we first employ the linear sampling method to reconstruct the shape, and then recover the impedance parameter by combining potential theory with the least squares method. Several numerical experiments are presented to demonstrate the feasibility and effectiveness of our methods.
\end{abstract}

\noindent\textbf{Keywords:}
elastic inverse scattering,
impedance boundary condition,
linear sampling method,
least squares method

\section{Introduction}
Inverse scattering problems have long been a central topic in applied mathematics and engineering, with profound implications in fields such as non-destructive testing, geophysical exploration, and medical imaging \cite{colton1998inverse}. These problems aim to reconstruct unknown scatterers or medium properties from measured scattered fields, presenting unique challenges due to their inherent ill-posedness and the complexity of wave-matter interactions. Among these, interior inverse scattering problems, where both the excitation source and measurements are confined  within the scattering domain, have gained increasing attention for applications in structural health monitoring of enclosed cavities or defect detection in bounded media. In recent years, various numerical methods have been proposed to address such inverse problems, including qualitative sampling methods \cite{ji2018direct,liu2019extended,zeng2021interior}, iterative quantitative methods \cite{chang2024novel,bao2019multifrequency,yin2023hybrid}, statistical inversion methods \cite{daza2017solution,palafox2017point,huang2021bayesian}, machine learning methods \cite{khoo2019switchnet,zang2020weak,raissi2019physics}.
	
	For interior inverse scattering in elasticity, existing studies have primarily focused on Dirichlet or Neumann boundary conditions \cite{Zeng2024NearfieldIM,Wang2025BayesianAF,ou2022interior,li2025factorization}, with limited results for impedance conditions. In contrast, impedance conditions introduce a coupling between the field and its conormal derivative, adding mathematical complexity to both the direct and inverse problems. This paper investigates the interior inverse scattering problem for an impenetrable elastic cavity with an impedance boundary condition, which offers a more realistic model of practical engineering surfaces \cite{chen2025shape}. The objective is to determine both the unknown cavity boundary and the impedance coefficient from scattered field data measured on an interior curve.
	We extend prior work on interior inverse acoustic scattering \cite{qin2012inverse2,Qin2013ReconstructionFC,qin2021reconstruction} to elasticity, which represents a non‑trivial generalization. Our study advances the theoretical understanding of interior inverse elastic scattering and provides a practical framework for reconstructing complex cavities with non-ideal boundaries and impedance parameters, with potential applications in engineering inspection and material science.
	
	The remainder of the paper is structured as follows. In section \ref{Problemf}, we formulate the interior
	scattering problem with impedance boundary conditions mathematically, prove a reciprocity property of the scattered elastic field, which is critical for uniqueness and numerical schemes. We also show that both the boundary and impedance parameter can be uniquely determined from the knowledge of incident elastic point sources and measurements on an interior curve. Section \ref{LMS} details the linear sampling method and its theoretical foundation for boundary reconstruction, while section \ref{impedance determine} focuses on reconstructing the impedance
	coefficient. We provide some preliminary numerical examples to demonstrate the feasibility of the proposed methods in section \ref{NE}. 

\section{Problem Formulation}\label{Problemf}
	Let $D \subset \mathbb{R}^2$ be a bounded, simply connected domain with $C^2$ continuous boundary $\partial D$. The medium inside the impenetrable elastic cavity $D$ is homogeneous and isotropic. Let $\omega > 0$ denote the frequency and $\beta$ impedance coefficient. The propagation of time-harmonic elastic waves in $D$ with an impedance boundary condition is governed by
	\begin{equation}
		\Delta^*u^s + \omega^2u^s = 0,\; x\in D,   \label{eq:(1)}
	\end{equation}
	where $u=u^s + u^i$ is the total field, and $u^s$ denotes the scattered field excited by an elastic point source $u^i(x,z,p)=E(x,z)p$ located at $z \in D$ with polarization $p \in \mathbb{S}:=\{x\in \mathbb{R}^2:|x|=1 \}$. 
	The Lamé operator $\Delta^*u$ and conormal derivative $\partial u/ \partial \nu$ are defined as
	\begin{equation*}
		\Delta^*u=\mu\Delta u+(\lambda+\mu) \nabla (\nabla \cdot u),
	\end{equation*}
	\begin{equation*}
		\frac{\partial u}{\partial \nu}=\lambda(\nabla\cdot u)\nu+\mu(\nabla u+\nabla u^T)\nu,
	\end{equation*}
	where $\lambda$ and $\mu$ are Lamé parameters, and $\nu$ is the unit outward normal vector to $\partial D$. Note that the conormal derivative has an intuitive physical meaning \cite{ammari2015mathematical}: 
      \begin{equation*}
		\frac{\partial u}{\partial \nu} = Tu,\ on\ \partial D.
	\end{equation*}
	In $\mathbb{R}^2$, the Green's tensor $E(x,z)$ of elastic scattering can be expressed as
	\begin{equation}
		\begin{aligned}
			E(x,z)=\frac{i}{4\mu}H^{(1)}_0(k_s|x-z|)I+\frac{i}{4w^2}\nabla^T_x \nabla_x(H^{(1)}_0(k_s(|x-z|))-H^{(1)}_0(k_p(|x-z|)),\ x\neq z,
		\end{aligned}
	\end{equation}
	where 
	\[
	k_p=\frac{w}{\sqrt{\lambda+2\mu}}, \quad  k_s=\frac{w}{\sqrt{\mu}},
	\]
	are the wavenumbers for compressional (P-wave) and shear (S-wave) components, respectively.
	The conormal derivative tensor $\partial E(x,z) / \partial \nu(x)$ is defined such that
	\begin{equation*}
		\frac{\partial E(x,z)}{\partial \nu} p = \frac{\partial E(x,z)p}{\partial \nu},
	\end{equation*}
	for all constant vectors $p$. It is easy to verify that
	\begin{equation} \label{SmmetrE}
		E(x,z) = E(z,x),   \quad
		\frac{\partial E(x,z)}{\partial \nu(x)}  = \frac{\partial E(z,x)}{\partial \nu(x)}.
	\end{equation}   
	
	For a given elastic point source $u^i$ and boundary $\partial D$, the direct scattering problem involves solving for $u^s$ in $D$ under the impedance boundary condition
	\begin{equation}
		\frac{\partial u}{\partial v}+i\beta u =0,\; x\in \partial D,  \label{eq:(2)}
	\end{equation}
	where $\beta(x) \in C(\partial D)$ is the impedance parameter. This paper addresses the inverse scattering problem of determining the unknown boundary $\partial D$ and the impedance $\beta(x)$ on the boundary from $u^s$ measured on a closed smooth curve $C\subset D$. Without loss of generality, $C$ is assumed to be a circle, and $\dot{C}$ denotes its interior. We adopt the following assumption in this paper.
	\begin{assumption}
		$\omega^2$ is not a Dirichlet eigenvalue of the negative Lamé operator in $\dot{C}$.    
	\end{assumption}
	This assumption is non-restrictive, as $C$ can be chosen such that it holds.
	
	\subsection{Uniqueness of the Inverse Problem}
	Using Betti’s formula \cite{arens2001linear}, the solution $u^s$ of \eqref{eq:(1)} can be expressed as
	\begin{equation} \label{eq:(5)}
		u^s(x,z,p)=\int_{\partial D} E(x,y)\frac{\partial u^s(y,z,p)}{\partial \nu(y)}-\frac{\partial E(x,y)}{\partial \nu(y)}u^s(y,z,p) \,ds(y),
	\end{equation}
	for $x \in D$, $z \in C$ and polarization $p \in \mathbb{S}$. We first prove a reciprocity relation for the scattered elastic field $u^s$ on $C$. 
	\begin{lemma}
		Assume $u^s(\cdot,z,p)$  satisfy equations \eqref{eq:(1)} and \eqref{eq:(2)}. For any $x, z \in C$ and polarizations $p,q \in \mathbb{S} $, we have  
		$$
        q \cdot u^s(x,z,p) = p \cdot u^s(z,x,q).
        $$
	\end{lemma}
	\begin{proof}
		From the integral presentation \eqref{eq:(5)} and the symmetric of Green's tensor \eqref{SmmetrE}, we expand $ q \cdot u^s(x, z, p) $ as
		\begin{align}
			\quad \quad q \cdot u^s(x,z,p) & = \int_{\partial D} q^T E(x,y) \frac{\partial u^s(y,z,p)}{\partial \nu(y)} - q^T\frac{\partial E(x,y)}{\partial \nu(y)}u^s(y,z,p) \,ds(y)  \nonumber
			\\
			& = \int_{\partial D} (E(y,x) q)^T  \frac{\partial u^s(y,z,p)}{\partial \nu(y)} - \left[\frac{\partial E(y,x)q}{\partial \nu(y)} \right]^T u^s(y,z,p) \,ds(y). \label{Usxz}
		\end{align}
		Using the boundary condition \eqref{eq:(2)}
		\[
		\frac{\partial u^s(x,z,p)}{\partial \nu(x)} + i\beta u^s(x,z,p) = - \frac{\partial E(x,z)p}{\partial \nu(x)}-i\beta E(x,z)p, 
		\]
		for $x \in \partial D$, $z \in C$. It follows that 
		\begin{align}
			\frac{\partial u^s(y,z,p)}{\partial \nu(y)} & = -i\beta u^s(y,z,p)-\frac{\partial E(y,z)p}{\partial \nu(y)}-i\beta E(y,z)p,	\label{eq:(22)} 
			\\
			\frac{\partial E(y,x)q}{\partial \nu(y)} & = -i\beta E(y,x)q -\frac{\partial u^s(y,x,q)}{\partial \nu(y)} -i\beta u^s(y,x,q),  \label{eq:(23)}
		\end{align}
		for $y \in \partial D$, $x, \ z \in C$ and $p,\ q \in \mathbb{S} $. 
		Substituting \eqref{eq:(22)} and \eqref{eq:(23)} into \eqref{Usxz}, we have 
		\begin{align}
			q \cdot u^s(x,z,p) 
			= \int_{\partial D} (E(y,x) q)^T  \left[-\frac{\partial E(y,z)p}{\partial \nu(y)}-i\beta E(y,z)p \right ] \nonumber\\ 
			+ \left[ \frac{\partial u^s(y,x,q)}{\partial \nu(y)} + i\beta u^s(y,x,q) \right]^T u^s(y,z,p) \,ds(y).  \label{Usxz1a} 
		\end{align}
		Similarly,   
		\begin{align}
			p \cdot u^s(z,x,q) 
			= \int_{\partial D} (E(y,z) p)^T  \left[-\frac{\partial E(y,x)q}{\partial \nu(y)}-i\beta E(y,x)q \right ] \nonumber \\
			+ \left[ \frac{\partial u^s(y,z,p)}{\partial \nu(y)} + i\beta u^s(y,z,p) \right]^T u^s(y,x,q) \,ds(y). \label{Uszx1b} 
		\end{align}
		As it is show in \cite{ou2022interior}, the third generalized Betti formula \cite{arens2001linear} yields that
		\begin{equation}
			\int_{\partial D} \left[ \frac{\partial u^s(y,x,q)}{\partial \nu(y)} \right]^T u^s(y,z,p)  - (u^s(y,x,q))^T \frac{\partial u^s(y,z,p)}{\partial \nu(y)} \,ds(y)=0,
			\label{eq:(6)}
		\end{equation}
		\begin{equation}
			\int_{\partial D} (E(y,x)q)^T \frac{\partial E(y,z)p}{\partial \nu(y)}-\left[ \frac{\partial E(y,x)q}{\partial \nu(y)} \right]^T E(y,z)p \,ds(y)=0,
			\label{eq:(7)}
		\end{equation}
		for $x, z \in C$ and polarization directions $p,q \in \mathbb{S}$. Substituting \eqref{eq:(6)} and \eqref{eq:(7)} into \eqref{Usxz1a}, we can further deduce \eqref{Usxz1a} to
		\begin{align}
			q \cdot u^s(x,z,p) 
			= \int_{\partial D} \left[- \frac{\partial E(y,x)q}{\partial \nu(y)} -i\beta (E(y,x) q) \right ]^T E(y,z)p  \nonumber \\
			+ \left[ \frac{\partial u^s(y,z,p)}{\partial \nu(y)} + i\beta u^s(y,z,p)  \right]^T u^s(y,x,q) \,ds(y). \label{Usxz2}
		\end{align}
		The proof is done by comparing \eqref{Uszx1b} and \eqref{Usxz2}.
	\end{proof}   

	In the following, we show the uniqueness for the inverse problem.
	
	\begin{theorem}
		Assume that the scattered fields $u^s(x,z,p)$ satisfy equations \eqref{eq:(1)} and \eqref{eq:(2)}, then the boundary $\partial D$ and impedance coefficient $\beta$ can be uniquely determined from $u^s(x,z,p)$ for all $x,z \in C$ and $p \in \mathbb{S}$.
	\end{theorem}
	
	\begin{proof}
		Step1 (Uniqueness of $\partial D$): 
		Assume $D_1 \neq D_2$ are distinct bounded Lipschitz domains containing $C$ with associated impedance coefficients $\beta_{1}$ and $\beta_{2}$ on the boundary, respectively. And $u^s_{1}$ and $u^s_{2}$ are the corresponding scattered fields satisfying equations \eqref{eq:(1)} and \eqref{eq:(2)} for $D_{1}$ and $D_{2}$, respectively. If
		\begin{equation*}
			u^s_{1}(x,z,p)=u^s_{2}(x,z,p),
		\end{equation*}
		for all $x, z\in C$ and $p\in \mathbb{S}$, then for $v=u^s_{1}-u^s_{2}$, it satisfies
		\begin{align*}
			\Delta^*v+\omega^2v=0,\quad x\in \dot{C},\\
			v=0,\quad x\in C.
		\end{align*}
		Since $\omega^2$ is not a Dirichlet eigenvalue in $\dot{C}$,  we have $v \equiv0 $ in $\dot{C} \bigcup C$.
		Let $D_0$ be an simply connected component of $D_{1}\bigcap D_{2}$ containing $C$, then analyticity implies that $v \equiv0 $ in $\overline{D_{0}}$, i.e.,
		$$u^s_{1}(x,z,p)=u^s_{2}(x,z,p), \ x \in \overline{D_{0}}, \ z \in C. $$
		By reciprocity relationship, we have
        $$
        e_i \cdot u^s_{1}(z,x,p) = p \cdot u^s_{1}(x,z,e_i) = p \cdot u^s_{2}(x,z,e_i) = e_i \cdot u^s_{2}(z,x,p),
        $$
		for $i=1, \ 2$, where $e_1 = (1,0)^T$ and $e_2 = (0,1)^T$ is an orthonormal basis on $\mathbb{S}$. Thus
		$$u^s_{1}(z,x,p)=u^s_{2}(z,x,p), \  z \in C, \ x \in \overline{D_{0}}. $$
		Using the same argument in the above steps, we obtain
		$$u^s_{1}(x,z,p)=u^s_{2}(x,z,p), \ x, z \in \overline{D_{0}}.$$
		
		Without loss of generality, there exists $x^* \in \partial D_{0}$ satisfying $x^* \in \partial D_{1}$ and $x^* \not\in \partial D_{2}$. Let $\nu(x^*)$ be the unit normal vector at $x^*$ pointing into the interior of $D_{1}$. Then for all $n\in\mathbb{N}^+$, the sequence defined by 
		$$x_{n} = x^* + \frac{1}{n}\nu(x^*)$$
		are entirely contained in $D_{0} \subset D_2$. In view of well posedness of the solution in cavity $D_2$,
		\begin{align*}
			\lim_{n \to \infty} \left[ \frac{\partial u^s_{1}(x^*,x_n,p)}{\partial \nu}+i\beta_{1}u^s_{1}(x^*,x_n,p)  \right] 
			& =
			\lim_{n \to \infty} \left[ \frac{\partial u^s_{2}(x^*,x_n,p)}{\partial \nu}+i\beta_{1}u^s_{2}(x^*,x_n,p)  \right]  \\
			& = 
			\frac{\partial u^s_{2}(x^*,x^*,p)}{\partial \nu}+i\beta_{1}u^s_{2}(x^*,x^*,p).
		\end{align*}
		On the other hand, the boundary condition on $\partial D_1$ indicates that
        \begin{equation*}
        \lim_{n \to \infty} \left[ \frac{\partial u^s_{1}(x^*,x_n,p)}{\partial \nu}+i\beta_{1}u^s_{1}(x^*,x_n,p)  \right] 
		= -\lim_{n \to \infty} \left[ \frac{\partial E(x^*,x_n)p}{\partial \nu}+i\beta_{1} E(x^*,x_n)p  \right]  
		= \infty.    
        \end{equation*}
		This yields a contradiction of the limit and thus $D_1 = D_2$, i.e., the boundary $\partial D$ is uniquely determined.
		
		Step2 (Uniqueness of $\beta$):
		With $D_{1}=D_{2}=D$, assume $\beta_{1}(x_0) \neq \beta_{2}(x_0)$ for $x_0 \in \partial D$, then by continuity, there exits an open arc $\Gamma_0 \subset U(x_0,\delta) \bigcap \partial D$, such that $\beta_1(x) \neq \beta_2(x) $ for all $x \in \Gamma_0$, where $U(x_0,\delta)$ is a neighborhood of $x_0$.
        If $u^s_{1}(x,z,p) = u^s_{2}(x,z,p)$ for all $x, \ z \in C$, following the same argument in step 1, we can conclude that
		$$u^s_{1}(x,z,p)=u^s_{2}(x,z,p),$$
		for all $x, z \in \overline{D}$. Let
		\begin{equation*}
			u=u^s_{1}(x,z,p)+E(x,z)p=u^s_{2}(x,z,p)+E(x,z)p, \quad x \in \overline{D}, z \in C.
		\end{equation*} 
		Using the boundary condition on $\partial D_1$ and $\partial D_2$, respectively, we have
		\begin{equation}\label{Total12}
			\frac{\partial u}{\partial \nu}+i\beta_{1}u=0, \quad \frac{\partial u}{\partial \nu}+i\beta_{2}u=0,
		\end{equation}
		on $\partial D_{1} = \partial D_{2} = \partial D$. This means that 
		\begin{equation}\label{Totalu}
			(\beta_{1}(x)-\beta_{2}(x))u(x,z,p) = 0,
		\end{equation}
		for all $x \in \partial D$. Since $\beta_{1}(x) \not\neq \beta_{2}(x)$ for all $x \in \Gamma_0 \subset \partial D$, we can obtain from \eqref{Total12} and \eqref{Totalu} that 
		\[
		u(x,z,p) = 0, \ \frac{\partial u(x,z,p)}{\partial \nu} = 0, \ x \in \Gamma_0.
        \]
		By Holmgren's uniqueness theorem, it follows that $u(x,z,p) \equiv 0$ for all $x \in D \setminus \{z\}$ and $z \in C$. Let $\nu(z)$ be the unit outward normal to $C$, then for the sequence
		\[
		x_{n} = z + \frac{1}{n} \nu(z), \ n = 1,2,\cdots
		\]
		in $D \setminus \{z\}$, it holds that $u(x_{n},z,p) = 0$, i.e., 
		$$u^s_{1}(x_{n},z,p)=-E(x_{n},z)p,$$
		for all $n \in \mathbb{N}^+$. Letting $n \rightarrow \infty$ on both sides, we have
		\[
		\lim_{n \to \infty} u^s_{1}(x_{n},z,p) = u^s_{1}(z,z,p),
		\]
		but the limit on right hand side is unbounded. The proof is done by contradiction, so that $\beta_{1}(x) = \beta_{2}(x)$ for all $x\in \partial D$.
	\end{proof}

	\section{The linear sampling method}\label{LMS}
	To reconstruct the boundary $\partial D$, we define the near field operator $N: \ [L^2(C)]^2\to [L^2(C)]^2$ as
	\begin{equation}
		(Ng)(x)=\int_{C} u^s(x,z,g(z)) \,ds(z),\  g\in [L^2(C)]^2,\ x\in C.\label{NF}
	\end{equation}
	Since $u^s$ is analytic, the operator $N$ is compact. In addition, we show an important property for $N$ in the following theorem.
	\begin{theorem} \label{PropertyN}
		The operator $N$ is injective and has dense range in $[L^2(C)]^2$.
	\end{theorem}
	\begin{proof}
		Let $Ng=0$. To prove injectivity, it suffices to show that $g=0$. Define 
		\begin{equation*}
			v(x)=\int_{C} u^s(x,z,g(z)) \,ds(z),\ x\in D. 
		\end{equation*}
		It's easy to verify that $v$ satisfies
		\begin{align*}
			\Delta^*v+\omega^2v &=0,\quad \text{in} \ \dot{C},\\
			v&=0,\quad \text{on} \ C.    
		\end{align*}
		Since $\omega^2$ is not a Dirichlet eigenvalue, we conclude that $v\equiv0$ in $\dot{C}$. By the unique continuation, we have $v\equiv0$ in $\overline{D}$. It then follows
		$v=0$ and $\frac{\partial v}{\partial \nu}=0$ on $\partial D$, which indicates $\frac{\partial v}{\partial\nu}+i\beta v =0$.
		For $x\in \mathbb{R}^2\setminus C$, let
		\begin{equation*}
			K(x)= \int_{C} E(x,z)g(z) \,ds(z).
		\end{equation*}
		It satisfies $\Delta^*K+\omega^2K=0$ in the exterior of $D$ and the Kupradze’s radiation condition. On $\partial D$, using the boundary condition for $u^s(x,z,g(z))$, we have
		\[
		\frac{\partial K}{\partial \nu}+i\beta K = - \frac{\partial v(x)}{\partial \nu(x)} - i\beta v(x) = 0.
		\]
		Since the solution of exterior impedance problem is unique \cite{athanasiadis2011boundary}, $K=0$ in $\mathbb{R}^2\setminus\overline{D}$.
		It follows from the unique continuation principle that $K=0$ in the exterior of $C$. 
		Using the jump relationship for the single potential
		\begin{equation*}
			K_{+}=K_{-}, \ \frac{\partial K_{+}}{\partial v}-\frac{\partial K_{-}}{\partial v}=-g, \ \text{on} \ C,
		\end{equation*}
		where $v$ is the unit outward normal to $C$, and $\pm$ denote the limits as $x$ approaches $C$ from exterior and interior of $C$, respectively. Since 
		\[
		K_{+} = 0, \  \frac{\partial K_{+}}{\partial v}=0,
		\]
		we can quickly get
		\[
		K_{-}=0,\  \frac{\partial K_{-}}{\partial v}=g \ \text{on} \ C.
		\]
		On the other hand, we have that $K(x)$ satisfies $\Delta^*K+\omega^2K =0$ in $\dot{C}$ with homogeneous Dirichlet condition on $C$.   
		Since $\omega^2$ is not a Dirichlet eigenvalue, we obtain $K=0$ in $C \bigcup \dot{C}$. Thus $g=0$, i.e., $N$ is injective.
		
		From the reciprocity relation, the $[L^2(C)]^2$ adjoint operator $N^*:[L^2(C)]^2 \to [L^2(C)]^2$ can be expressed as
		\begin{equation*}
			(N^*\psi)(x)= \overline{ \int_{C} u^s(x,z,\overline{\psi(z)}) \,d s(z) }, \  \psi\in [L^2(C)]^2,\ x\in C.
		\end{equation*}
		That is $(N^*\psi)(x) = \overline{(Ng)(x)}$ with $g(z)=\overline{\psi(z)}$. Therefore, the injectivity of $N^*$ is obtained by the injectivity of $N$. And since $\mathcal{R}(N)^\perp=\mathcal{N}(N^*) = 0$, we conclude that $N$ has dense range in $[L^2(C)]^2$.
	\end{proof}
	
	Consider the general impedance boundary problem 
    \begin{align}
    \Delta^*u^s+\omega^2 u^s &=0,\qquad x\in D,  \label{eq:(10)}    \\
    \frac{\partial u^s}{\partial \nu}+i\beta u^s &=f,\qquad x\in\partial D,  \label{eq:(11)}
    \end{align}
	where $f\in [H^{-\frac{1}{2}}(\partial D)]^2$.
	The weak formulation of problems  $\eqref{eq:(10)}$ and\ $\eqref{eq:(11)}$ is to find $u^s \in [H^1(D)]^2$ such that
	\begin{equation}
		\begin{aligned}
		&\lambda \int_{D} \bigtriangledown \cdot \ u^s \bigtriangledown \cdot \overline{v} \,dx + \frac{\mu}{2}\int_{D} (\bigtriangledown u^s+(\bigtriangledown u^s)^T) : (\bigtriangledown \overline{v} +
				(\bigtriangledown\overline{v})^T) \,dx\\
		&~~~~~~~~~~~~~~~~~~~~ -\int_{D} w^2 u^s \cdot \overline{v} \,dx + i\beta\int_{\partial D} u^s \cdot \overline{v} \,ds 
		=\int_{\partial D} f \cdot \overline{v} \,ds, \label{eq:formulation}
		\end{aligned}
	\end{equation}
	for all $v \in [H^1(D)]^2$. From \cite{kaiafa2021interior}, we have known that there exists a unique solution of $\eqref{eq:(10)}$ and\ $\eqref{eq:(11)}$ satisfying
	\[
	\| u^s \|_{H^1(D)} \leq M \|f\|_{H^{-\frac{1}{2}}(\partial D)},
	\]
	for given $f \in [H^{-\frac{1}{2}}(\partial D)]^2$ and positive constant $M$. Let  
	\[
	w_{g}(x)=\int_{C} E(x,z) g(z) \,ds(z),\ x\in \mathbb{R}^2\setminus C,
	\]
	with $g\in [L^2(C)]^2$, we define the operator $S: [L^2(C)]^2 \rightarrow [H^{-\frac{1}{2}}(\partial D)]^2$ 
	\[
	Sg=(\frac{\partial w_{g}}{\partial \nu}+i\beta w_{g})|_{\partial D},
	\]
	and the linear operator $B:[H^{-\frac{1}{2}}(\partial D)]^2 \to [L^2(C)]^2$ which maps the boundary value $f$ to the corresponding solution $u^s$ on $C$. Then the near field operator $N$ defined in $\eqref{NF}$ can be factorized as $N=-BS$.    
	
	\begin{theorem}
		The operator $B$ is injective, compact, and has dense range in $[L^2(C)]^2$.
	\end{theorem}
	\begin{proof}
	Assume $Bf=0$, i.e., the solution $u^s$ of \eqref{eq:(10)} satisfies $u^s=0$ on $C$. Since $\omega^2$ is not an interior Dirichlet eigenvalue in $\dot{C} \subset D$, we have $u^s\equiv0$ in $\dot{C}$. By the unique continuation principle and trace theorem, we have $u^s\equiv0$ in $D$ and $u^s|_{\partial D} = 0$.
	Then the weak formulation \eqref{eq:formulation} reduces to 
	\[
	\int_{\partial D} f \cdot \overline{v} \,ds = 0,
	\]
	for all $v \in[H^{\frac{1}{2}}(\partial D)]^2$. It follows that $f=0$ on $\partial D$, which means that the operator $B$ is injective.  

    Let
	\[
	\varphi(x)=\sum_{m=-n}^{n} \alpha_{m}J_{m}(k_{p}r) e^{im\theta}, \
	\psi(x)=\sum_{m=-n}^{n} \beta_{m}J_{m}(k_{s}r) e^{im\theta},
	\]
	where $(r,\theta)$ are polar coordinates of $x$, $r=|x|$, and $\alpha_m,\ \beta_m \in l^2$. Define
	\[
	w(x)= \nabla \varphi + curl \psi,
	\]
	with $\nabla v=(\partial_{x}v,\partial_{y}v)^T$ and $curlv=(\partial_{y}v,-\partial_{x}v)^T$. Direct calculating show us that $w(x)$ satisfies $\eqref{eq:(10)}$ and $\eqref{eq:(11)}$ with $f=(\frac{\partial w}{\partial \nu}+i\beta w)|_{\partial D}\in [H^{-\frac{1}{2}}(\partial D)]^2$. 
    Let $(e_r, e_\theta)$ denote the local unit orthogonal basis on $C$ in the polar coordinate system, we have 
	\begin{equation*}
		\begin{aligned}
		& \nabla \varphi = \sum_{m=-n}^{n}\alpha_{m}\left[k_{p}J'_{m}(k_{p}r) e_r+\frac{im}{r}J_{m}(k_{p}r) e_{\theta}\right] e^{im\theta},\\
		&curl \psi = \sum_{m=-n}^{n}\beta_{m}\left[\frac{im}{r}J_{m}(k_{s}r) e_{r}-k_{s}J'_{m}(k_{s}r) e_{\theta}\right] e^{im\theta}, 
		\end{aligned}
	\end{equation*}
	so that
    \begin{equation*}
    w|_C = \sum_{m=-n}^{n} \left[\alpha_{m} k_{p}J'_{m}(k_{p}r) + \beta_m \frac{im}{r} J_{m}(k_{s}r) \right] e^{im\theta} e_r 
	 + \sum_{m=-n}^{n} \left[\alpha_{m} \frac{im}{r} J_{m}(k_{p}r) - \beta_m k_s J'_{m}(k_{s}r) \right] e^{im\theta} e_\theta.    
    \end{equation*}
	Then the density of the range of $B$ follows directly from the completeness of exponential polynomials in $[L^{2}[0,2\pi]]^{2}$.
		
	To show the compactness. We choose a disc $\Omega$ such that $\dot{C} \subset \Omega \subset D$ and denote the boundary of $\Omega$ by $\Gamma$. Using the integral representation formula for $u^s$ in $\Omega$, we can decompose the operator $B$ as $B=B_{1}B_{2}$, where $B_{2}:[H^{-\frac{1}{2}}(\partial D)]^2 \to C(\Gamma)^2 \times C(\Gamma)^2$ is defined by
	\begin{equation*}
		(B_{2}f)(x)=(\frac{\partial u^s}{\partial \nu}|_{\Gamma},u^s|_{\Gamma}), 
	\end{equation*}
	and $B_{1}:C(\Gamma)^2 \times C(\Gamma)^2 \to [L^2(C)]^2$ is defined by
	\begin{equation}
		B_{1}(f_{1},f_{2})(x)=\int_{\Gamma} E(x,y) f_{1}(y) - \frac{\partial E(x,y)}{\partial \nu}  f_{2}(y)\,ds(y),\qquad x\in C.\nonumber
	\end{equation}
	Then the interior regularity of problems $\eqref{eq:(10)}$ and $\eqref{eq:(11)}$ ensures the boundedness of $B_{2}$. And since $B_{1}$ is compact, it follows that the composition $B=B_{1}B_{2}$ is compact.
	\end{proof}
	
	\begin{theorem}\label{RangeB}
	For given artificial polarization $q \in \mathbb{S}$, the function $E(x,s)q$ lies in the range of operator $B$ for $x\in C$ and $s \in \mathbb{R}^2 \setminus C$ if and only if $s \in \mathbb{R}^2\setminus{\overline D}$.
	\end{theorem}
	\begin{proof}
	If\ $s \in \mathbb{R}^2 \setminus{\overline D}$, we have the function $E(\cdot,s)q$ satisfies problem $\eqref{eq:(10)}$ and $\eqref{eq:(11)}$ with boundary value
	\[
	f(x) = \frac{\partial E(x,s)q}{\partial \nu} + i\beta E(x,s)q, \ x\in \partial D.
	\]
	It is obvious that $(Bf)(x) = E(x,s)q$ for $x\in C$, i.e., $E(x,s)q$ lies in the range of $B$.           
		
	If $s \in {\overline D} \setminus C$ and $E(x,s)q$ lies in the range of $B$ for $x\in C$, then it follows from the definition of $B$ that $E(\cdot,s)q$ is solution of $\eqref{eq:(10)}$ and $\eqref{eq:(11)}$ with 
	\[
	f(x) = \frac{\partial E(x,s)q}{\partial \nu} + i\beta E(x,s)q, \ x\in \partial D.
	\]
	However, $E(\cdot,s)q \not\in [H^1(D)]^2$. The proof is done by contradiction.
	\end{proof}

	\begin{theorem}
		The operator $S$ is bounded, injective and has dense range in $[H^{-\frac{1}{2}}(\partial D)]^2$.
	\end{theorem}
	\begin{proof}
	By defination, the operator $S$ is bounded.
		
	Now we will show the injectivity. Let $Sg=0$, it is sufficient to show that $g=0$. Note that the function $w_{g}(x)$ satisfies
    \begin{align*}
    &\Delta^*w_{g}+\omega^2w_{g}=0,\qquad x\in \mathbb{R}^2\setminus{\overline D},\\
		&\frac{\partial w_{g}}{\partial\nu}+i\beta w_{g}=0,\qquad x\in \partial D,    
    \end{align*}
	and the Kupradze's radiation condition. Following the same argument as in the proof of Theorem \ref{PropertyN}, we can obtain $g = 0$ on $C$. Thus, the operator is injective.
		
	To prove denseness, let $h \in [H^{-\frac{1}{2}}(\partial D)]^2$ satisfy $(Sg,h)=0$ for all $g$, i.e.,
	\begin{align*}
	0&=\int_{\partial D}\int_{C} \left[\frac{\partial E(x,z)}{\partial \nu(x)}+i\beta E(x,z) \right]g(z) \,ds(z)\overline{h(x)}ds(x)\\
	&=\int_{C} \int_{\partial D} \left[\frac{\partial E(x,z)}{\partial \nu(x)}+i\beta E(x,z) \right]\overline{h(x)} \,ds(x)g(z) \,ds(z) .   
	\end{align*}
	Thus
	\[
	\int_{\partial D} \left[\frac{\partial E(x,z)}{\partial \nu(x)}+i\beta E(x,z) \right]\overline{h(x)} \,ds(x) = 0,
	\]
	for all $z\in C$. Define
	\[
	v(z) = \int_{\partial D} \left[\frac{\partial E(x,z)}{\partial \nu(x)}+i\beta E(x,z) \right]\overline{h(x)} \,ds(x), \ z \in \mathbb{R}^2 \setminus \partial D.
	\]
	then $v$ satisfy $\Delta^*v+\omega^2 v = 0$ in $\dot{C}$ and the homogeneous Dirichlet boundary condition on $C$. 
	Since $\omega^2$ is not a Dirichlet eigenvalue in $\dot{C}$, we have $v \equiv 0 $ in $\dot{C}\cup C$. By analyticity, $v=0$ in $\overline{D}$. Applying the jump relation, we have
	\[
	v_+ - v_- = \overline{h}, \quad 
	\frac{\partial v_{+}}{\partial\nu}-\frac{\partial v_{-}}{\partial\nu} = -i\beta\overline{h},
	\]   
	where $\pm$ denote the limits as $x$ approaches $\partial D$ from exterior and interior of $D$, respectively. Then the conditions $v_{-} = \frac{\partial v_{-}}{\partial v}=0$ imply $\frac{\partial v_{+}}{\partial \nu}+i\beta v_{+}=0$  on $\partial D$. 
	 
	Now consider $v$ in the exterior of $D$. It can be verified that $v$ satisfies $\Delta^* v+\omega^2 v =0$ in $\mathbb{R}^2 \setminus {\overline D}$ and the Kupradze's radiation condition. Uniqueness for the exterior impedance problem for Navier equation reads that $v\equiv0$ in the exterior of $D$, then we have $v_+ = 0$ on $\partial D$. Therefore, $h = 0$, which show us that the operator $S$ has dense range in $[H^{-\frac{1}{2}}(\partial D)]^2$.
	\end{proof}
	
	To determine the boundary $\partial D$, we introduce the near field equation		
	\begin{equation}
		\int_{C} u^s(x,z,g_s(z)) \,ds(z)=E(x,s)q, \label{eq:(8)}
	\end{equation} 
	for all $x \in C$. Here, $s$ is the sampling point and $q \in \mathbb{S}$ is an artificial polarization. Using the near field operator $N$ defined in \eqref{NF}, we can present the integral equation \eqref{eq:(8)} as
	\[
	(Ng_s)(x) = E(x,s)q, \quad  x \in C,
	\]
	where the operator $N$ has decomposition $N=-BS$.
	The basic idea of the linear sampling method is to solve the near field equation for indicator function $g_s \in [L^2(C)]^2$.   
	For $s \in \mathbb{R}^2 \setminus \overline{D}$, we can see that if $g_s$ is a solution for the near field equation, then $u^s$ and the elastic point source $E(x,s)q$ satisfy the same boundary condition, i.e.,
    \begin{equation} \label{BasicLSM}
	\frac{\partial u^s}{\partial \nu} +i \beta u^s = -\frac{\partial E(x,s)q}{\partial \nu} -i \beta E(x,s)q, 
	\end{equation}
	on $\partial D$. As $s \rightarrow \partial D$, we have that the right hand of \eqref{BasicLSM} blows up, and hence the left hand blows up. Thus $\|g_s\|_{L^2(C)} \rightarrow \infty $ and this behavior determines $\partial D$. The above argument is only heuristic because of the ill-posedness of \eqref{eq:(8)}. However, we will show that the near field equation admits an approximate solution that behaves similarly as $s \rightarrow \partial D$. The following theorem is of fundamental importance to the linear sampling method for the inverse scattering problem we considered in the paper.
	\begin{theorem}\label{MainThm}
	For sampling point $s \in \mathbb{R}^2 \setminus C$ and artificial polarization $q \in \mathbb{S}$.
	\begin{enumerate}
		\item [1)] If $s \in \mathbb{R}^2\setminus\overline{D}$, then for any $\epsilon>0$, there exists a solution $g^{\epsilon}_s \in [L^2(C)]^2$ satisfying
		\begin{equation*}
			\Vert Ng^{\epsilon}_{s}(x)-E(x,s)q\Vert_{L^2(C)}<\epsilon.  
		\end{equation*}\\
		 For a fixed $\epsilon$, we have
		\begin{equation}\label{MainThm1}
			\lim_{s \to \partial D} \Vert S g^{\epsilon}_{s} \Vert_{H^{-\frac{1}{2}}(\partial D)} = \infty,
		\end{equation}
		and 
		\begin{equation}\label{MainThm2}
			\lim_{s \to \partial D}\Vert g^{\epsilon}_{s}\Vert_{L^2(C)}=\infty.
		\end{equation}
		\item [2)] If $s \in D\setminus C$, then for every $\epsilon>0$, one can find a solution $g^{\epsilon}_s \in [L^2(C)]^2$ such that
		\[
		\Vert Ng^{\epsilon}_{s}(x) - E(x,s)q \Vert_{L^2(C)} < \epsilon.
		\]
		Furthermore, we have
		\begin{equation}\label{MainThm3}
			\lim_{\epsilon \to 0} \Vert S g^{\epsilon}_{s} \Vert_{H^{-\frac{1}{2}}(\partial D)} = \infty,
		\end{equation}
		and 
		\begin{equation}\label{MainThm4}
			\lim_{\epsilon \to 0} \Vert g^{\epsilon}_{s} \Vert_{L^2(C)} = \infty.
		\end{equation}
		\end{enumerate}
	\end{theorem}
	\begin{proof}
	\begin{enumerate}
	\item [1)] Let $s \in \mathbb{R}^2\setminus\overline{D}$. By Theorem \ref{RangeB}, there there exists
	  $f_s \in [H^{-\frac{1}{2}}(\partial D)]^2$ such that $Bf_s = - E(x,s)q$ for $x \in C$. Since $S$ has dense 
        range in $[H^{-\frac{1}{2}}(\partial D)]^2$ and $B$ is bounded, it follows that for every $\epsilon_{0}>0$, there exists a function $ g_{\epsilon_{0}}^{s}\in [L^2(C)]^2$ satisfying
		\begin{equation}\label{eq:(15)}
			\Vert Sg^{\epsilon_{0}}_{s}-f_{s} \Vert_{H^{-\frac{1}{2}}(\partial D)}<\epsilon_{0}. 
		\end{equation}
		and then 
		\[
		\Vert BSg^{\epsilon_{0}}_{s}-Bf_{s} \Vert_{L^2(C)}<c \epsilon_{0},
		\]
		where $c$ is a positive constant. Let $\epsilon = c \epsilon_{0}$, we have
		\[
		\Vert Ng^{\epsilon}_{s} - E(x,s)q \Vert_{L^2(C)}< \epsilon. 
		\]
		By the definition of $B$, it is obvious that 
		\[
		f_s(x) = -\frac{\partial E(x,s)q}{\partial \nu} - i \beta E(x,s)q, \ x\in \partial D.
		\]
		As $s \to \partial D$, $\Vert f_{s} \Vert_{H^{-\frac{1}{2}}(\partial D)}$ blows up due to the singularity of $E(x,s)$. Then for a fixed $\epsilon$, from \eqref{eq:(15)}, we have \eqref{MainThm1} holds. And thus \eqref{MainThm2} also holds.
			
		\item [2)]  Now let $s \in D\setminus C$, Theorem \ref{RangeB} implies that $E(\cdot,s)q$ is not in the range of the operator $B$. However, $B$ has dense range in $[L^2(C)]^2$. Therefore, applying Tikhonov regularization to the equation $(Bf_s)(\cdot)=-E(\cdot,s)q$, we obtain, for every $\epsilon > 0$, there exists a unique regularized solution $f_s^{\alpha} \in [H^{-\frac{1}{2}}(\partial D)]^2$ given by
		\begin{equation*}
			f_s^{\alpha} = -\sum_{j=1}^{\infty}\frac{\mu_{j}}{\alpha+\mu_j^2}(E(\cdot,s)q,y_{j})x_{j}, 
		\end{equation*}
		where $(\mu_{j},x_{j},y_{j})$ is a singular system for the compact operator $B$ such that
		\[
		\Vert (Bf_s^{\alpha})(\cdot) + E(\cdot,s)q\Vert_{L^2(C)} < \frac{\epsilon}{2},
		\]
		and by Picard's theorem we have $\Vert f_s^{\alpha} \Vert_{H^{-\frac{1}{2}}(\partial D)} \to \infty$ as $\alpha \to 0$. Since $S$ has dense range, there exists $g_s^{\alpha}$ such that	
		\begin{equation}\label{eq:(16)}
			\Vert S g^{\alpha,\epsilon}_{s}  -f^\alpha_{s}\Vert_{H^{-\frac{1}{2}}(\partial D)}<\frac{\epsilon}{2c}, 
		\end{equation}
		where $c>0$ is a constant such that $\Vert B\| < c$. So that,
		\begin{align*}
			& \Vert N g^{\alpha,\epsilon}_{s}(\cdot) - E(\cdot,s)q\Vert_{L^2(C)} \\   
			&  \ \ \ \ \ \ \ \ \ \leq \Vert BS g^{\alpha,\epsilon}_{s} - Bf^\alpha_{s}\Vert_{L^2(C)}+\Vert Bf^\alpha_{s}(\cdot) + E(\cdot,s)q\Vert_{L^2(C)} < \epsilon.
		\end{align*}
		Since $\lim_{\epsilon \to 0}\alpha(\epsilon)=0$, we have
		\[
		\lim_{\epsilon \to 0}\Vert f_s^{\alpha(\epsilon)}\Vert_{H^{-\frac{1}{2}}(\partial D)}=\infty,
		\]
		then \eqref{eq:(16)} indicates that \eqref{MainThm3} holds. And since the operator $S$ is bounded, and therefore \eqref{MainThm4} also holds.
		\end{enumerate}
	\end{proof}
	
	\section{Determination of the surface impedance parameter}\label{impedance determine}
	Given $\partial D$ from linear sampling, we reconstruct $\beta$ following the decomposition method in \cite{zeng2013decomposition}. Let $B_{\rho}$ be a disc large enough such that $D \subset B_\rho$, we define the single layer potential by
	\begin{equation*}
		(L\phi)(x)=\int_{\partial  B_{\rho}} E(x,y)\phi(y) \, ds(y), \ x \not\in \partial B_\rho. 
	\end{equation*}
	Since the kernel $E(x,y)$ is analytic, $L$ is compact. In addition, we have the following property.
	
	\begin{theorem}
		The operator $L$ is compact, injective and has dense range in $[L^2(C)]^2$.
	\end{theorem}
	\begin{proof}
	To prove injectivity, assume $L\phi=0$ on $C$. Then it is sufficient to show that $\phi = 0$. Let
	\begin{equation*}
		w(x)=\int_{\partial B_{\rho}} E(x,y)\phi(y) \,ds(y), \ x\in \mathbb{R}^2\setminus\partial B_{\rho},
	\end{equation*}
	we have $w(x)$ satisfies \eqref{eq:(1)} in $\dot{C}$ with homogeneous Dirichlet boundary value on $C$. Since $\omega^2$ is not a Dirichlet eigenvalue in $\dot{C}$, we obtain $w=0$ in $\dot{C}\cup C$. Applying the unique continuation principle, it yields that $w=0$ in $\overline{B_{\rho}}$. From the jump relations for the single layer potential, we obtain
	\[
	w_{+}=w_{-}, \  \frac{\partial w_{+}}{\partial \nu}-\frac{w_{-}}{\partial \nu}=-\phi,
	\]
	on $\partial B_{\rho}$. Since we have know that $w_{-}=\frac{\partial w_{-}}{\partial \nu}=0$ on $\partial B_{\rho}$, then $w_{+}=0$ and $\frac{\partial w_{+}}{\partial \nu}=-\phi$ on $\partial B_{\rho}$. Consider $w$ in the exterior of $B_\rho$, it is a radiating solution of \eqref{eq:(1)} in the exterior of $B_\rho$ with homogeneous Dirichlet boundary value. Since the radiating solution is unique \cite{knops1981three}, $w=0$ in $\mathbb{R}^2\setminus\overline{B_{\rho}}$. Then we have $\frac{\partial w_{+}}{\partial \nu}=-\phi=0$, hence $\phi = 0$, which means $L$ is injective.
		
	In the following, we show that $L$ has dense range in $[L^2(C)]^2$. The $L^2$ adjoint operator $L^*:[L^2(C)]^2 \to [L^2(\partial B_{\rho})]^2$ is given by 
	\[
	(L^*\psi)(x) = \overline{ \int_{C} E(x,y) \overline{\psi(y)} \, ds(y)}, 
	\] 
	for all $\psi\in [L^2(C)]^2$ and $x\in \partial B_{\rho}$. We prove that the operator $L^*$ is injective. Let $L^*\psi = 0$ and define
	\begin{equation*}
		w_0(x)=\overline{ \int_{C} E(x,y) \overline{\psi(y)} \, ds(y)},\ x\in \mathbb{R}^2\setminus C, 
	\end{equation*} 
	then
    \begin{align*}
    \Delta^*w_0+\omega^2w_0 &=0,\qquad \mathbb{R}^2\setminus \overline{B_{\rho}},\\
    w_0 &=0,\qquad x \in \partial B_{\rho},
    \end{align*}
	and the Kupradze's radiation condition. Following the same argument as in the proof of Theorem \ref{PropertyN}, it is straightforward to show that the operator $L^*$ is likewise injective. Since $\mathcal{R}(L)^\perp = \mathcal{N}(L^*)$, we can conclude that $L$ has dense range\ $[L^2(C)]^2$.
	\end{proof}
	
	Given the scattered field $u^s$ on $C$, we can set up the integral equation
	\begin{equation}\label{eq:(17)}
		(L\phi)(x)=u^s(x,z,p), \  x,z\in C  
	\end{equation}
	for unknown density $\phi\in [L^2(\partial  B_{\rho})]^2$. Due to the ill-pose nature, we employ the classical Tikhonov regularization to solve it$\eqref{eq:(17)}$. In particular, we
	look for the approximation solution $\phi_\alpha$ by solving the following regularized problem
	\[
	(\alpha I+L^*L)\phi =L^*u^s(\cdot,z,p),\ z \in C,
	\]
	where $\alpha>0$ is the regularization parameter and $L^*$ is the adjoint operator of $L$. 
	
	Once we have found $\phi_\alpha$, we can compute the approximation scattered field by $u_\alpha^s = L \phi_\alpha$. Then the impedance parameter on $\partial D$ is determined by identifying $\beta$ such that 
	the impedance boundary condition is approximately satisfied, i.e.,
	\[
	\frac{\partial u_\alpha}{\partial \nu} + i \beta u_\alpha \approx 0,
	\]
	where the total field and the conormal derivative on the boundary $\partial D$ are evaluated by
	\[
	u_\alpha = u_\alpha^s + u^i = \int_{\partial B_\rho} E(x,y)\phi_\alpha(y) \,ds(y) + u^i,
	\] 
	and
	\[
	\frac{\partial u_\alpha(x)}{\partial \nu(x)} = \int_{\partial B_\rho} \frac{\partial E(x,y)}{\partial \nu(x)}\phi_\alpha(y) \,ds(y) + \frac{\partial u^i(x)}{\partial \nu(x)},
	\]
	with incident field $u^i = E(x,z)p$. 
	
	Numerically, we use the least squares method to determine the surface impedance $\beta$. Let $\xi_{n}$, $n=0,1,2,...,N$ be some basis functions, we present the impedance function $\beta$ by 
	\begin{equation*}
		\beta(x)=\sum_{n=0}^{N} c_{n}\xi_{n}(x),\ x\in\partial D. 
	\end{equation*}
	Then we can determine the coefficients $c_n, \ n=0,1,2,\cdots,N$ by solving the following
	minimization problem  
	\[
	\min_{c} \sum_{m=1}^M \left| \frac{\partial u_\alpha(x_m)}{\partial \nu} + i \sum_{n=0}^{N} c_{n}\xi_{n}(x_m) u_\alpha(x_m) \right|^2
	\]
	for $x_m, \ m=1,2,\cdots,M$ on $\partial D$, and $c=(c_0,c_1,c_2,\cdots,c_N)$.

    \section{Numerical Examples}\label{NE}

    In this section we present some numerical examples using synthetic data to verify the viability of our methods. Based on the classical potential theory, we know that the single-layer potential
	\begin{equation*}
		u^s(x)=\int_{\partial D}  E(x,y)\phi(y)\,ds(y),\ x\in \mathbb{R}^2\setminus\partial D,
	\end{equation*}
	with continuous density $\phi$. Then $u^s$ is a solution of the interior impedance problem $Eqs.\eqref{eq:(1)}$ and $Eqs.\eqref{eq:(2)}$ in $D$, provided that $\phi(\cdot)=\phi(\cdot,z)$ satisfies the integral equation
	\begin{equation}
		\int_{\partial D} \frac{\partial E(x,y)}{\partial\nu(x)} \phi(y)\,ds(y)+i\beta\int_{\partial D} E(x,y) \phi(y)\,ds(y)+\frac{1}{2}\phi(x) 
		=- \left( \frac{\partial E(x,z)p}{\partial\nu(x)} + i\beta E(x,z)p \right)\label{DirectUs}
	\end{equation}
	for $x\in\partial D$ and $z\in C$. In the numerical examples, we use Nystr$\ddot{o}$m method to solve the integral equation \eqref{DirectUs} for synthetic measurement data $u^s(x,z,p)$ with $x,z\in C$ and polarizations $e_1 = (1,0)^T$ and $e_2 = (0,1)^T$.
	
	After we have recorded the synthetic data on $C$, we turn to the problem of solving
	the linear ill-posed integral equation $\eqref{eq:(8)}$. Using the reciprocity relation, we can rewrite \eqref{eq:(8)} as two scalar equations
	\begin{equation}\label{LsmNumer}
		\int_{C} u^s(z,x,e_l) \cdot g_s(z) \,ds(z) = e_l \cdot (E(x,s)q), \ l=1,2,
	\end{equation}  
	for all $x,z \in C$. Here $s \not\in C$ is the sampling point and $q \in \mathbb{S}$ is artificial polarization. The two scale integrals will be approximated by a quadrature.
	Approximating the two scale integrals in \eqref{LsmNumer} by a quadrature rule. Then the fully discrete problem corresponding to \eqref{LsmNumer} is to find $g_s(z_j) \in \mathbb{C}^2, \ j=1,2,\cdots, N_C$ satisfying
    \[
    \sum_{j=1}^{N_C} \omega_j u^s(z_j, x_m, e_l) \cdot g_s(z_j) = e_l \cdot (E(x_m,s)q)
    \]
    for $l=1,2$ and $m=1,2,\cdots,N_C$. Sampling point $s$ is selected within the region between two rectangular grids $B_{1}$ and $B_{2}$, satisfying $C\subset B_{2} \subset D \subset B_{1}$. 
    We may rewrite the above linear system as
    \begin{equation}\label{LsmNumer1}
    A g_s = b_s,
    \end{equation} 
    where $A$ is a $2N_C \times 2N_C$ matrix, $g_s \in \mathbb{C}^{2N_C}$ is the vector of unknowns, and $b_s$ is the right-hand side depending on the sampling point $s$. Since $A$ comes from a compact operator, we choose to employ the Tikhonov regularization method solving for $g_s$ from \eqref{LsmNumer1}.  

    In numerical simulations, we choose the cavity $D$ with impedance coefficient $\beta$ to be one of the following.
    \begin{enumerate}
    \item an ellipse with parameters given by
    \begin{align*}
    &\partial D = \{x|x = (1.8\cos t,2.6\sin t), \ 0 \leqslant t \leqslant 2\pi \}, \\
    &\beta (t)=2+\sin^2 t,\ 0 \leqslant t \leqslant 2\pi.
    \end{align*}
    \item a triangle with parameters  given by
    \begin{align*}
    &\partial D = \{x|x = (2+0.3\cos 3t)(-\sin t,\cos t), \ 0 \leqslant t \leqslant 2\pi \},\\
    &\beta (t)=2+\cos (t-\pi),\ 0\leqslant t\leqslant 2\pi.
    \end{align*}
    \end{enumerate}

For all examples, we choose the curve $C$ to be a circle with radius $r_C = 0.4$ and centered at the origin. $80$ measurement and point source locations are uniformly distributed on $C$. We consider the pressure point source and the shear point source, respectively, i.e., polarizations $e_1 = (1,0)^T$ and $e_2 = (0,1)^T$. The curve $ \partial B_\rho$  is the circle with radius $r_{B_\rho} = 4.5$. We adopt a global relative noise to contaminate synthetic data matrix \(u^s(x_i,z_j,p)\) for $i,j=1,2,\cdots,80$. In precise, for given noise level \(\delta\), the noisy data are generated as
\[
u^s_{\delta} = u^s + \delta \cdot \frac{\|u^s\|_{\mathrm{F}}}{\|E\|_{\mathrm{F}}} \, E,
\]
where \(\|\cdot\|_{\mathrm{F}}\) denotes the Frobenius norm, and \(E\) is a complex Gaussian random matrix whose real and imaginary parts are independently drawn from the standard normal distribution. 

To reconstruct the boundary $\partial D$ via the linear sampling method, we choose a uniform grid of the sampling region $\Omega:=[-4,4] \times [-4,4] \backslash \ \dot{C}$. After solving \eqref{LsmNumer1} for each sampling point $s \in \Omega$, we compute $\|g_s\|$ in $\Omega$ and plot the heat map.
According to Theorem \ref{MainThm}, $\|g_s\|$ (where $\|\cdot\|$ denotes$\|\cdot\|_{L^2}$) should become large as $s \to \partial D$ from outside of $D$ and $s \in D \setminus C$. 
In Figure $\ref{FIG:0.01ty}(a)$, we show the heat map of the ellipse determined from the measurement data under noise level $\delta = 1\% $, where we also give the exact cavity for comparison.

In order to determine the surface impedance $\beta$, we should guess the boundary of the cavity according
to the numerical results obtained using the linear sampling method. In the following, we choose the reconstructed boundary to be the contour line such that $\|g_s\| = C_g$ for some constant $C_g$, which is determined from the reconstruction of a reference circle. In our examples, we choose 
\[
C_g = 0.6 \max_{s_i}\|g_{s_i}\|,
\]
for all sampling points $s_i \in \Omega$. 
Figure $\ref{FIG:0.01ty}(b)$ compares the approximate boundary of the elliptical cavity with its exact boundary. 
With the approximate boundary obtained in the above, the surface impedance $\beta$ is then recovered via the least squares method following Section \ref{impedance determine}. We choose the basis functions $\xi_{n}=e^{int}$, $(n=0,1,2,...,N)$, where the truncation index $N$ is selected by trial and error. In our examples, we choose $N=2$. Figure $\ref{FIG:0.01ty}(c)$ shows the true impedance and the approximate impedance for the elliptical cavity.
\begin{figure}[!htp]
\centering
\begin{subfigure}[t]{0.4\textwidth}
\centering
\includegraphics[height=5cm]{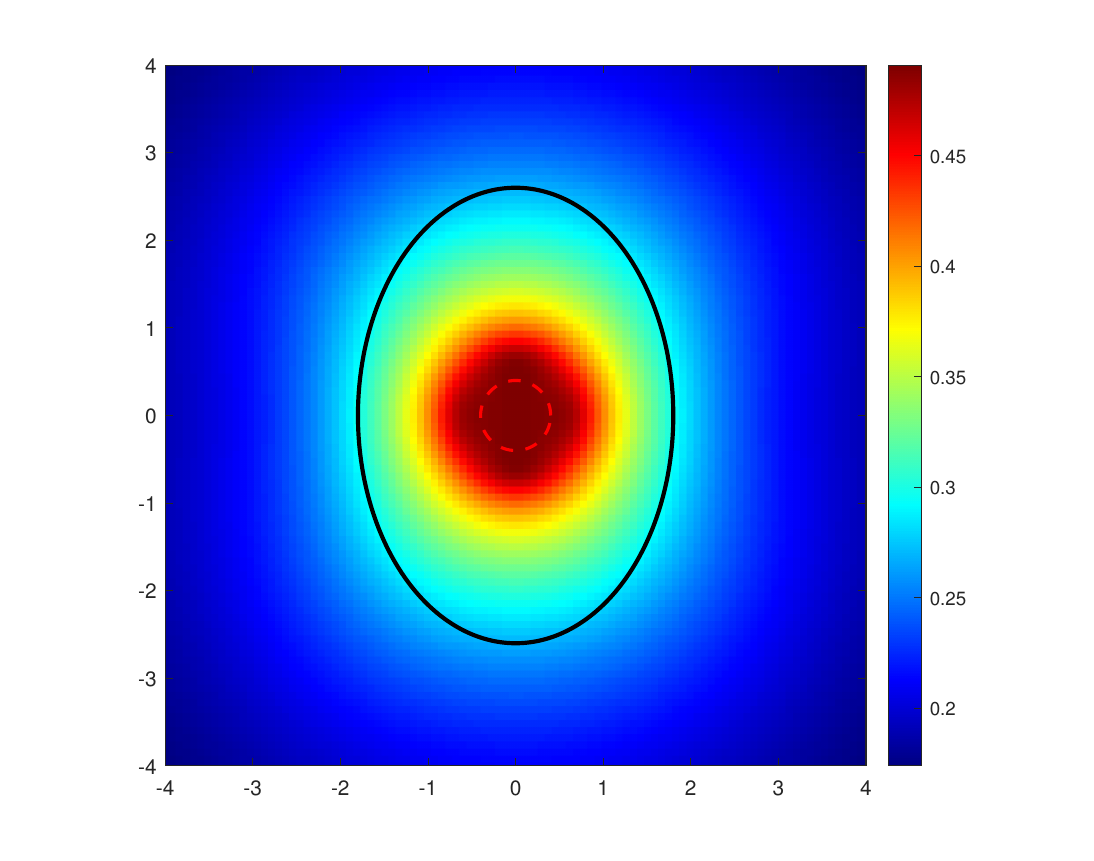}
\caption{The heat map of $\|g_s\|$}
\label{FIG:ellipse map 0.01}
\end{subfigure}\hfil
\begin{subfigure}[t]{0.4\textwidth}
\centering
\includegraphics[height=5cm]{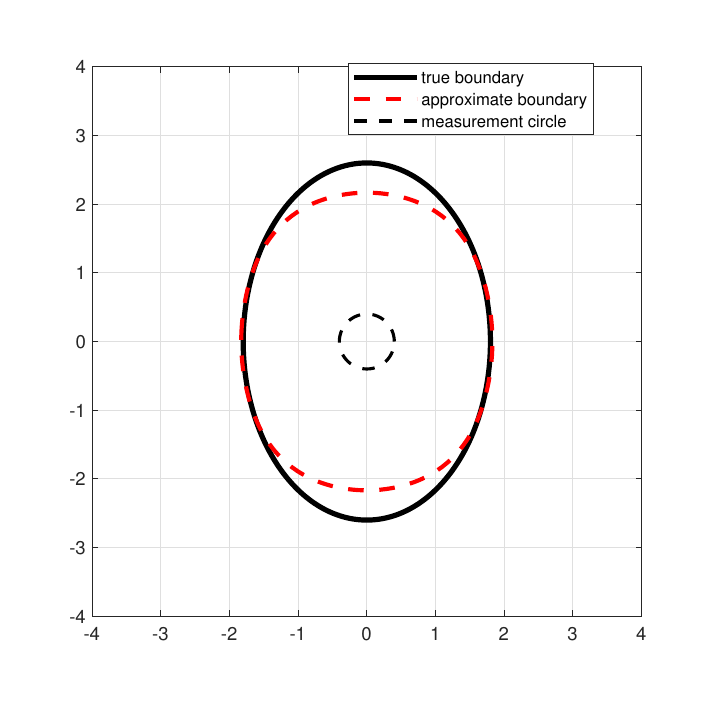}
\caption{Reconstructed result for $\partial D$}
\label{FIG:ellipse boundary 0.01}
\end{subfigure}\hfil
\begin{subfigure}[t]{0.4\textwidth}
\centering
\includegraphics[height=5cm]{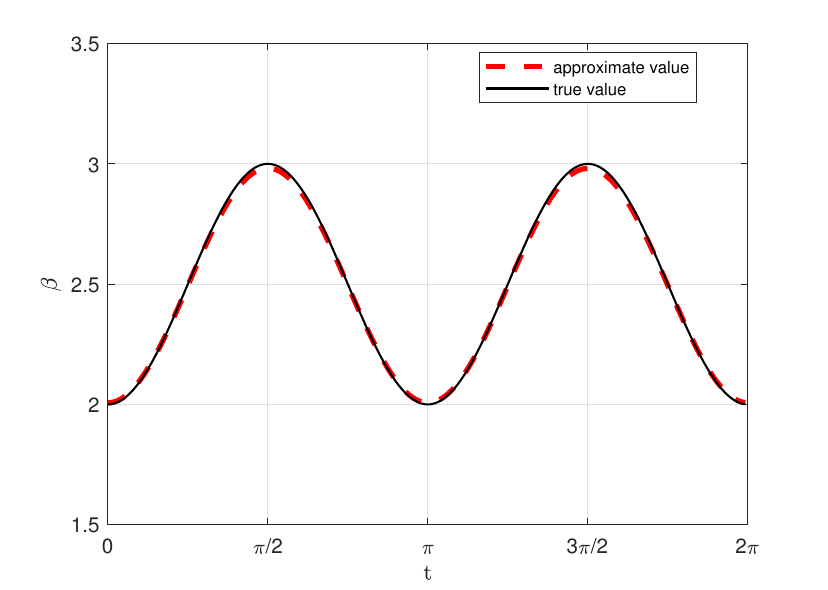}
\caption{Reconstructed result for $\beta$}
\label{FIG:ellipse impedance 0.01}
\end{subfigure}
\caption{Reconstructed results of elliptical cavity with $\omega=\sqrt{14},\lambda=1,\mu=1,\delta = 1\%$.}
\label{FIG:0.01ty}
\end{figure}
Figure $\ref{FIG:0.01ty}$ shows the successful reconstruction of the boundary and surface impedance of the elliptical cavity with $\delta=1\%$ noise. To test whether the method remains effective when the noise level increases, we also considered the reconstruction of the elliptical cavity with $\delta=5\% $ noise. The corresponding results are displayed in Figure $\ref{FIG:0.05ty}$, which indicate that the method exhibits relative stability against measurement noise.
\begin{figure}[!htp]
\centering
\begin{subfigure}[t]{0.4\textwidth}
\centering
\includegraphics[height=5cm]{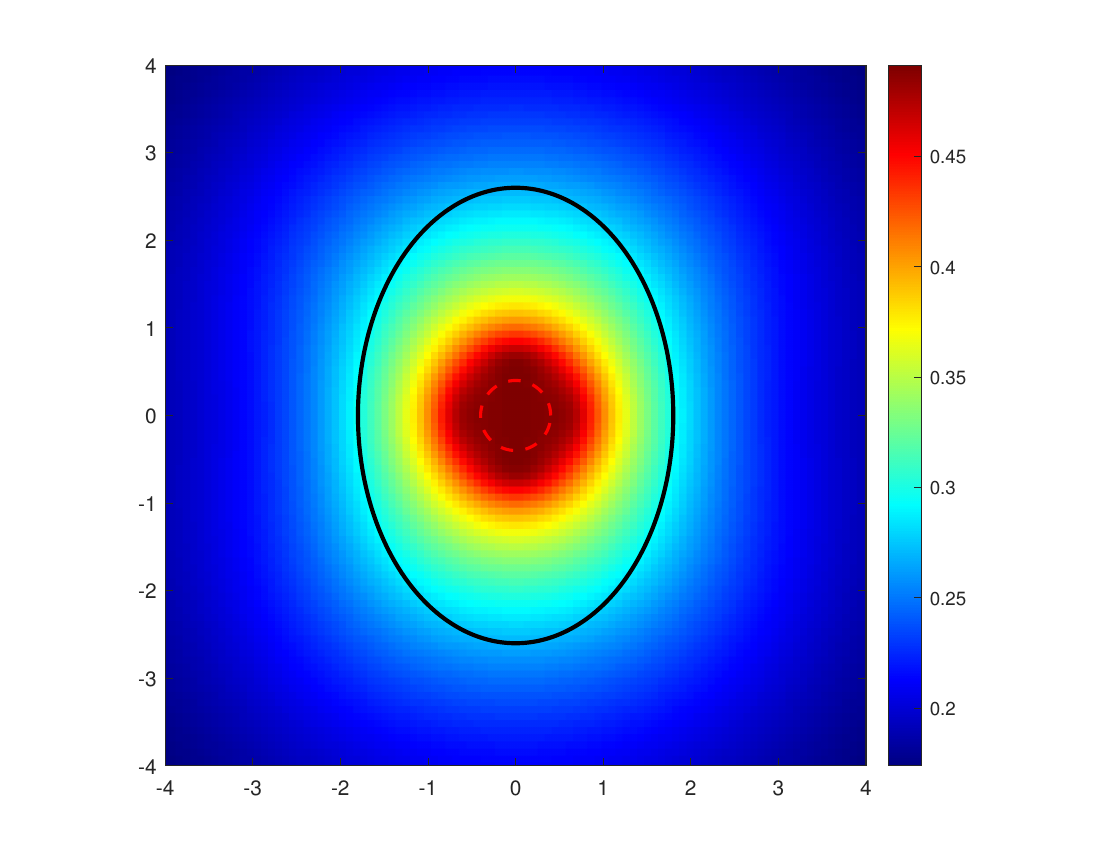}
\caption{The heat map of $\|g_s\|$}
\label{FIG:ellipse 0.05}
\end{subfigure}\hfil
\begin{subfigure}[t]{0.4\textwidth}
\centering
\includegraphics[height=5cm]{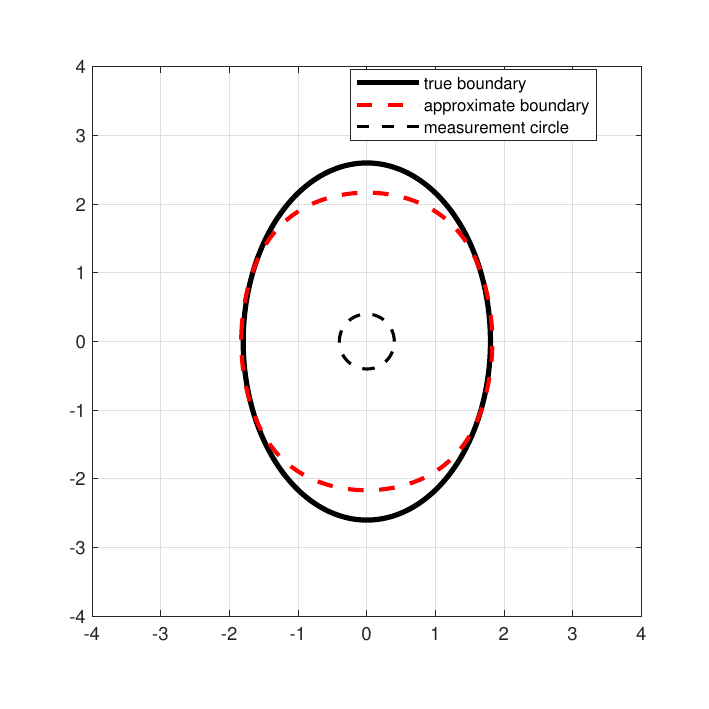}
\caption{Reconstructed result for $\partial D$}
\label{FIG:ellipse boundary 0.05}
\end{subfigure}
\begin{subfigure}[t]{0.4\textwidth}
\centering
\includegraphics[height=5cm]{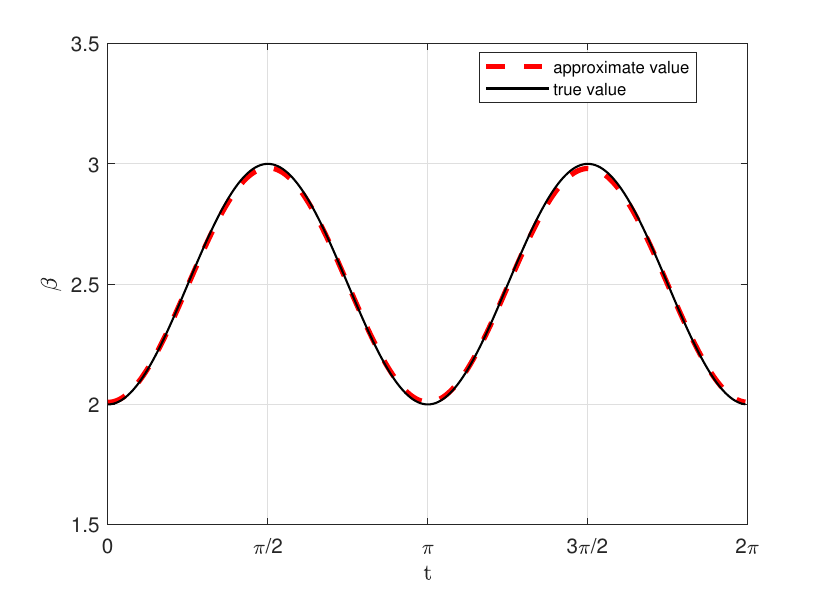}
\caption{Reconstructed result for $\beta$}
\label{FIG:ellipse impedance 0.05}
\end{subfigure}
\caption{Reconstructed results of elliptical cavity with $\omega=\sqrt{14},\lambda=1,\mu=1,\delta=5\%$.}
\label{FIG:0.05ty}
\end{figure}

For the second example, we consider the reconstruction of a triangular cavity that its surface impedance distribution differs from that of the elliptical cavity. The results are showed in Figure \ref{FIG:0.01sj} and Figure \ref{FIG:0.02sj} with different noise level. Considering the added noise, the result is satisfactory.
\begin{figure}[!htp]
\centering
\begin{subfigure}[t]{0.4\textwidth}
\centering
\includegraphics[height=5cm]{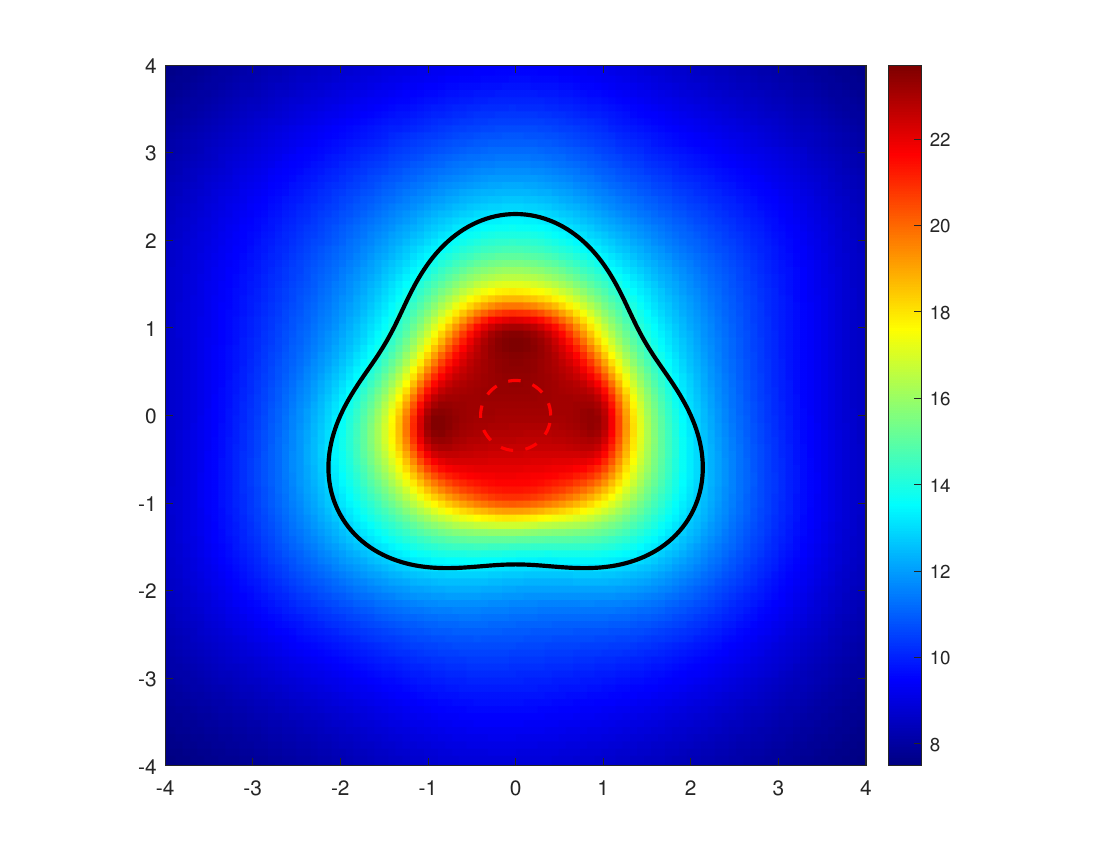}
\caption{The heat map of $\|g_s\|$}
\label{FIG:triangle 0.01}
\end{subfigure}\hfil
\begin{subfigure}[t]{0.4\textwidth}
\centering
\includegraphics[height=5cm]{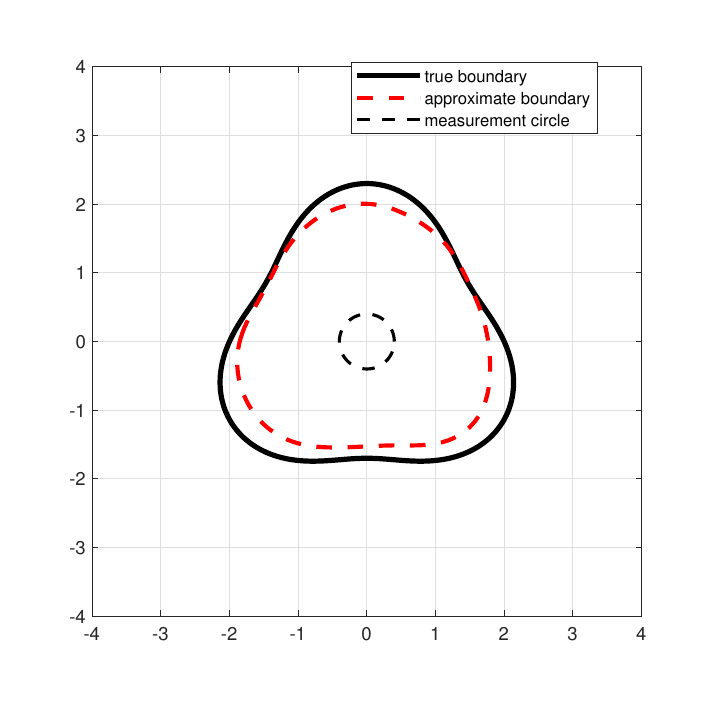}
\caption{Reconstructed result for $\partial D$}
\label{FIG:triangle boundary 0.01}
\end{subfigure}
\begin{subfigure}[t]{0.4\textwidth}
\centering
\includegraphics[height=5cm]{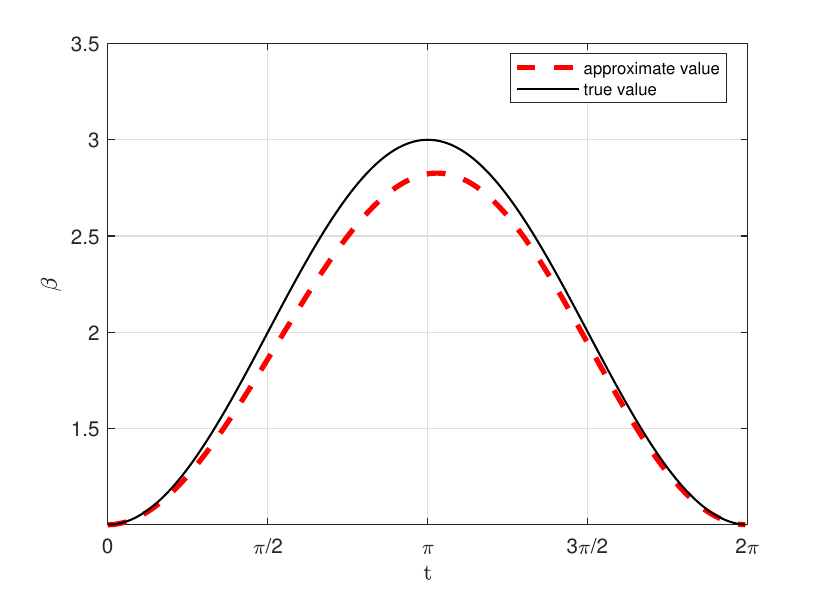}
\caption{Reconstructed result for $\beta$}
\label{FIG:triangle impedance 0.01}
\end{subfigure}
\caption{Reconstructed results of triangular cavity with $\omega=3,\lambda=1,\mu=1,\delta=1\%$. }
\label{FIG:0.01sj}
\end{figure}
\begin{figure}[!htp]
\centering
\begin{subfigure}[t]{0.5\textwidth}
\centering
\includegraphics[height=5cm]{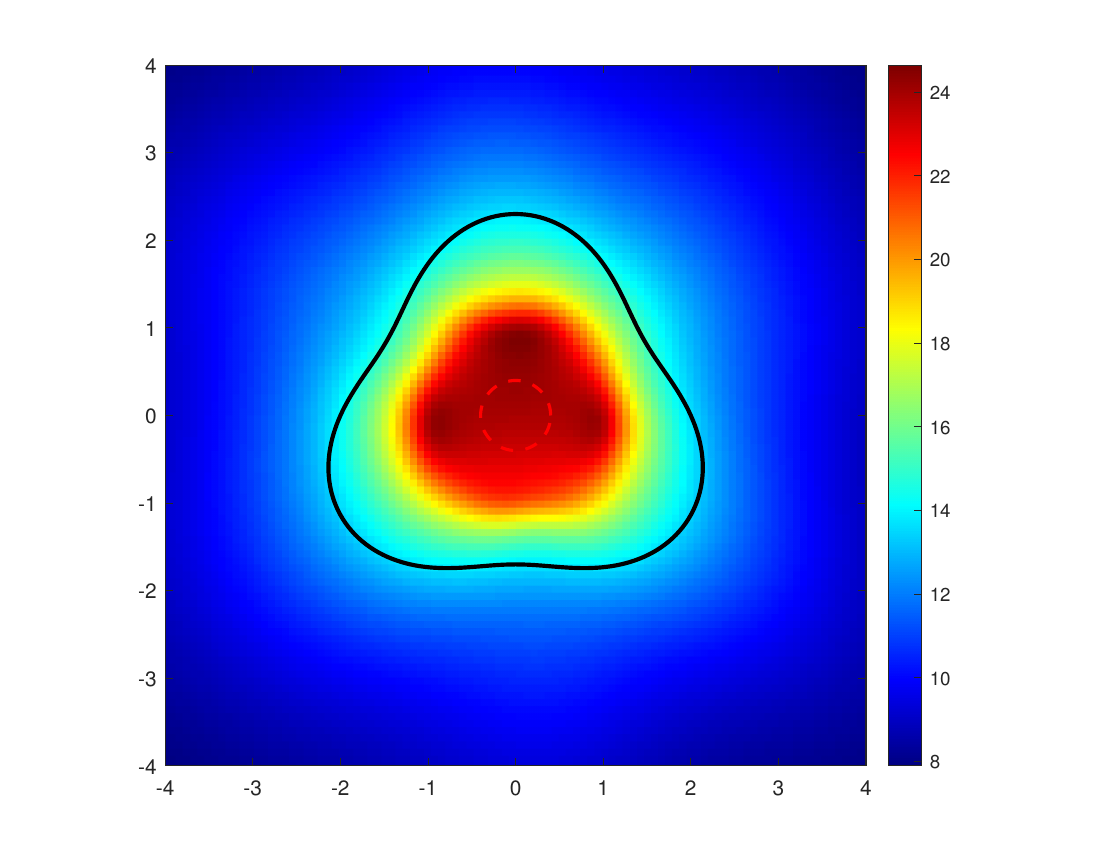}
\caption{The heat map of $\|g_s\|$}
\label{FIG:triangle 0.02}
\end{subfigure}\hfil
\begin{subfigure}[t]{0.4\textwidth}
\centering
\includegraphics[height=5cm]{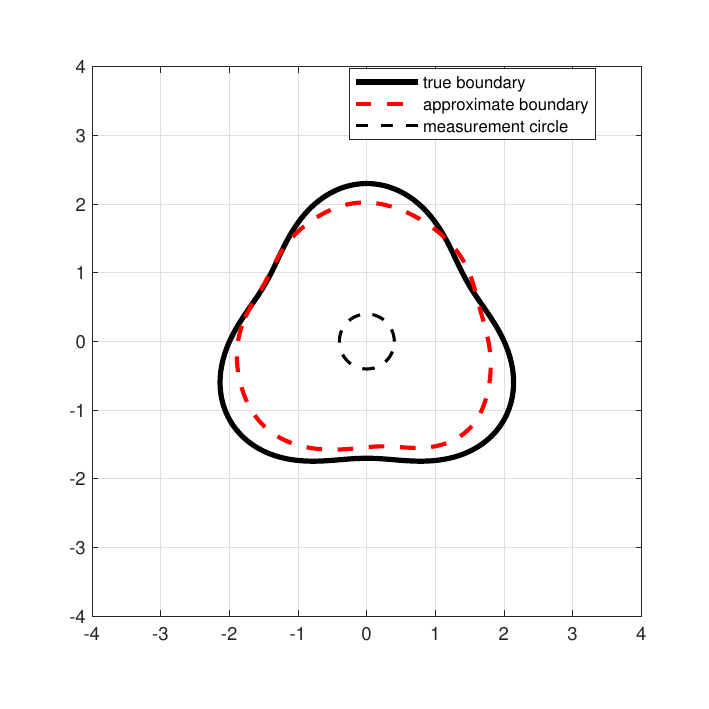}
\caption{Reconstructed result for $\partial D$}
\label{FIG:triangle boundary 0.02}
\end{subfigure}
\begin{subfigure}[t]{0.4\textwidth}
\centering
\includegraphics[height=5cm]{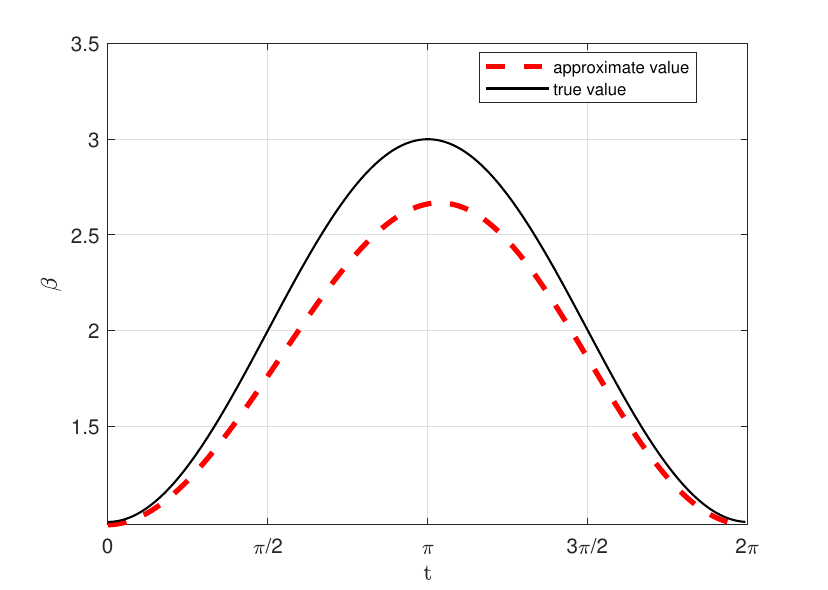}
\caption{Reconstructed result for $\beta$}
\label{FIG:triangle impedance 0.02}
\end{subfigure}
\caption{Reconstructed results of triangular cavity with $\omega=3,\lambda=1,\mu=1,\delta=2\%$.}
\label{FIG:0.02sj}
\end{figure}

\section*{Funding}

The work of Zeng is supported by the Science and Technology Research Program of Chongqing Municipal Education Commission (Grant No. KJQN202200525).

\section*{Author contributions}
Shiyu Xu: Formal analysis; Investigation; Software; Visualization; Writing – original draft. 
Huiling Zheng: Data curation; Validation; Writing – review \& editing.
Fang Zeng: Conceptualization; Funding acquisition; Methodology; Project administration; Resources; Supervision; Writing – review \& editing.
% List author names and the contributions made to the article, using terms from the NISO Contributor Roles Taxonomy (CRediT) https://credit.niso.org
\section*{Data availability}
The data that support the findings of this study are available from the corresponding author upon reasonable request.

\printbibliography

\end{document}